\documentclass{amsart}[12pt]
\usepackage[dvipsone]{graphicx}
\usepackage{psfrag}
\usepackage{epsfig,amsmath,amssymb,enumerate}
\usepackage[all]{xy}

\newcommand{\lga}{\longrightarrow}
\newcommand{\lgaf}{\longleftarrow}

\newcommand {\Z}{\mathbb Z}
\newcommand {\R}{\mathbb R}
\newcommand {\N}{\mathbb N}

\newcommand {\F}{\mathbb F}

\newcommand {\PP}{\mathcal P}
\newcommand {\QQ}{\mathcal Q}

\newcommand {\sub}{\subset}

\newtheorem{theorem}{Theorem}[section]
\newtheorem{lemma}[theorem]{Lemma}

\newtheorem{proposition}[theorem]{Proposition}
\newtheorem{corollary}[theorem]{Corollary}

\theoremstyle{definition}
\newtheorem{definition}[theorem]{Definition}

\theoremstyle{remark}
\newtheorem{remark}[theorem]{Remark}

\numberwithin{equation}{section}

\begin{document}

\title{Weak ${\mathcal Z}$-structures and one-relator groups}

%%%%%%%%%%%%%%%%%%%%%%%%%%%%%%%%%%%  authors' information  %%%%%%%%%%%%%%%%%%%%%%%%

\author{M. C\'ardenas}
\address{Departamento de Geometr\'{\i}a y Topolog\'{\i}a, Fac. Matem\'aticas, Universidad de Sevilla,
C/. Tarfia s/n 41012-Sevilla, Spain} \email{mcard@us.es}
\author{F. F. Lasheras}
\address{Departamento de Geometr\'{\i}a y Topolog\'{\i}a, Fac. Matem\'aticas, Universidad de Sevilla,
C/. Tarfia s/n 41012-Sevilla, Spain} \email{lasheras@us.es}
\author{A. Quintero}
\address{Departamento de Geometr\'{\i}a y Topolog\'{\i}a, Fac. Matem\'aticas, Universidad de Sevilla,
C/. Tarfia s/n 41012-Sevilla, Spain} \email{quintero@us.es}


\date{\today}

% at present the "communicated by" line appears only in ERA and PROC
%\commby{}

\begin{abstract} Motivated by the notion of boundary for hyperbolic and $CAT(0)$ groups, Bestvina \cite{B} introduced the notion of a (weak) $\mathcal Z$-structure and (weak) $\mathcal Z$-boundary for a group $G$ of type $\mathcal F$ (i.e., having a finite $K(G,1)$ complex), with implications concerning the Novikov conjecture for $G$. Since then, some classes of groups have been shown to admit a weak $\mathcal Z$-structure (see \cite{Gui1} for example), but the question whether or not every group of type $\mathcal F$ admits such a structure remains open. In this paper, we show that every torsion free one-relator group admits a weak $\mathcal Z$-structure, by showing that they are all properly aspherical at infinity; moreover, in the $1$-ended case the corresponding weak $\mathcal Z$-boundary has the shape of either a circle or a Hawaiian earring depending on whether the group is a virtually surface group or not. Finally, we extend this result to a wider class of groups still satisfying a Freiheitssatz property.
\end{abstract}

\maketitle

%%%%%%%%%%%%%%%%%%%%%%%%%%%%%%%%%%%%%%%%%%%%%%%%%%%%%%%%%%%%%%%%%%%%%%%%%%%%%%%%%%%%%%
\section{Introduction}
%%%%%%%%%%%%%%%%%%%%%%%%%%%%%%%%%%%%%%%%%%%%%%%%%%%%%%%%%%%%%%%%%%%%%%%%%%%%%%%%%%%%%%
We recall that a compact metrizable space $W$ is a compactification of a (path connected) locally compact metrizable space $X$ if it contains a homeomorphic copy $X \sub W$ as a dense open subset. Furthermore, we say that $W$ is a $\mathcal Z$-compactification of $X$ if $Z = W - X$ is a $\mathcal Z$-set in $W$, i.e., for every open set $U \sub W$ the inclusion $U - Z \hookrightarrow U$ is a homotopy equivalence; equivalently (if $X$ is an ANR), there is a homotopy $H : W \lga W$ with $H_0=id_W$ and $H_t(W) \sub X$ for all $t >0$. And in this case, we say that $Z$ is a $\mathcal Z$-boundary for $X$. The model example is that in which $W$ is a compact manifold and $Z \subseteq \partial W$ is a closed subset.\\
\indent From now on all spaces will be ANRs; in fact, we will deal with (connected) locally finite CW-complexes.\\
\indent In \cite{B} Bestvina introduced the notion of $\mathcal Z$-structure and $\mathcal Z$-boundary for a group (of type $\mathcal F$) as an attempt to generalize the already existing notion of boundary for hyperbolic and $CAT(0)$ groups; namely,
\begin{definition} \label{structure} A $\mathcal Z$-structure on a group $G$ is a pair $(W, Z)$ of spaces satisfying:
\begin{itemize}
\item[(1)] $W$ is a contractible ANR,
\item[(2)] $Z$ is a $\mathcal Z$-set in $W$,
\item[(3)] $X = W - Z$ admits a proper, free and cocompact action by $G$, and
\item[(4)] (nullity condition) For any open cover of $W$, and any compact subset $K \subseteq X$, all but finitely many translates of $K$ lie in some element of the cover.
\end{itemize}
Observe that is not necessary that $W$ be finite-dimensional, by \cite{Moran, Moran2}. If only conditions (1)-(3) are satisfied, then $(W,Z)$ is called a weak $\mathcal Z$-structure on $G$, and we refer to $Z$ as a (weak) $\mathcal Z$-boundary for $G$.
\end{definition}
\indent An additional condition can be added to the above:
\begin{itemize}
\item[(5)] The action of $G$ on $X$ can be extended to $W$.
\end{itemize}
If conditions (1)-(5) are satisfied, then $(W,Z)$ is called a $\mathcal{EZ}$-structure on $G$. It was shown in \cite{FL} that the Novikov conjecture holds for any (torsion free) group admitting an $\mathcal{EZ}$-structure. Examples of groups admitting an $\mathcal{EZ}$-structure are (torsion free) $\delta$-hyperbolic and $CAT(0)$ groups, see \cite{BM, B}.\\
\indent Although not stated explicitly, it follows that such a group $G$ as in Definition \ref{structure} must be of type $\mathcal F$ (see \cite[Prop. 1.1]{Gui1}). The more conditions on a $\mathcal Z$-structure on a group the better; nonetheless, a weak $\mathcal Z$-boundary already carries significant information about the group and, when it exists, it is well-defined up to shape, and it is always a first step towards finding a stronger structure on it. The question whether or not every group of type $\mathcal F$ admits a (weak) $\mathcal Z$-structure still remains open. In this paper, we give a positive answer to this question for a class of groups containing all torsion free one-relator groups. For this, we first use some previous work from \cite{CLQR2, LasRoy} to show that they are all properly aspherical at infinity, and then combine this with some recent work in \cite{CLQR_JPAA} to characterize the corresponding boundary in the $1$-ended case. Our main results are Theorems \ref{thm1} and \ref{thm2} below.
\begin{theorem} \label{thm1} Every finitely generated, torsion free one-relator group $G$ admits a weak $\mathcal Z$-structure. Moreover, if $G$ is $1$-ended then the corresponding weak $\mathcal Z$-boundary has the shape of either a circle or a Hawaiian earring depending on whether $G$ is a virtually surface group or not.
\end{theorem}
Bestvina \cite{B} already pointed out that the Baumslag-Solitar group $\langle x,t; t^{-1}x t = x^2 \rangle$ admits a $\mathcal Z$-boundary homeomorphic to the Cantor-Hawaiian earring, which is shape equivalent to the ordinary Hawaiian earring (see \cite{MarSe} as a general reference for shape theory). It is worth mentioning that for the particular class of Baumslag-Solitar groups a stronger $\mathcal{EZ}$-structure has been recently described in \cite{Gui2}.\\

More generally, we may construct the following class $\mathcal C$ of finitely presented groups starting off from one-relator group presentations as follows. Let $G$ and $H$ be finitely
generated one-relator groups, and
assume ${\PP} = \langle X; r \rangle$ and ${\QQ} = \langle Y ; s
\rangle$ are presentations of $G$ and $H$ with a single relation, respectively. Let
$V \sub X$ and $W \sub Y$, together with a bijection $\eta : V \lga
W$, be (possibly empty) subsets not containing all the generators involved in any of
the relators of the corresponding presentation. We declare the
corresponding amalgamated product $G*_F H$ associated with $\eta$
(over a free group of rank $card(V)=card(W)$) to be in our class $\mathcal C$ together with the obvious presentation for it obtained from ${\PP}$ and ${\QQ}$. It seems
natural to consider the class ${\mathcal C}$ of all finitely
generated one-relator groups together with those finitely presented
groups which can be obtained by successive applications of the construction above, so that ${\mathcal C}$ is closed under
amalgamated products (over free subgroups) of the type just described. The group presentations obtained in this way still satisfy a Freiheitssatz property, see \cite{LasRoy} for more details. Henceforth, we will refer to those groups as in $\mathcal C$ as ``generalized" one-relator groups.
\begin{remark} \label{ex1} The first interesting examples of groups in the class
${\mathcal C}$ (other than one-relator groups and their free products) are those groups $G$
given by a presentation of the form $\langle X; r,s \rangle$, where
$r,s$ are cyclically reduced words so that $r \in F(Y)$, $Y
\subsetneq X$, and $s \in F(X) - F(Y)$ misses at least one generator
in $Y$ which occurs in $r$. Indeed, one can obtain $G$ as an
amalgamated product $\langle Y ; r \rangle *_F \langle (X - Y) \cup
Z ; s \rangle$ over a free group of rank $card(Z)$, where $Z
\subsetneq Y$ is the subset consisting of those generators in $Y$
which occur in $s$.
\end{remark}
\begin{remark} From the above remark, one can see that
Higman's group $H$, with presentation $\langle a,b,c,d ;
a^{-2}b^{-1}ab, b^{-2}c^{-1}bc, c^{-2}d^{-1}cd, d^{-2}a^{-1}da
\rangle$, is in ${\mathcal C}$. Indeed, $H$ can be expressed as an
amalgamated product (of the type described above) of two copies of
$\langle x,y,z ; x^{-2}y^{-1}xy, y^{-2}z^{-1}yz \rangle$ over a free
subgroup of rank $2$ (see \cite{H}).
\end{remark}
\begin{remark} It was also shown in \cite{LasRoy} that every finitely presented group
$G$ given by a staggered presentation ${\PP}$ is in ${\mathcal C}$. We recall that $\langle X ; R \rangle$ is defined
to be a {\it staggered} presentation if there is a subset $X_0 \sub
X$ so that both $R$ and $X_0$ are linearly ordered in such a way
that: (i) each relator $r \in R$ contains some $x \in X_0$; (ii) if
$r, r'$ are relators with $r < r'$, then $r$ contains some $x \in
X_0$ that precedes all elements of $X_0$ occurring in $r'$, and $r'$
contains some $y \in X_0$ that comes after all those occurring in $r$
(see \cite{LS}).
\end{remark}
Theorem \ref{thm1} above together with the results in \cite{LasRoy} yield the following generalization.
\begin{theorem}  \label{thm2} Every torsion free, $1$-ended generalized one-relator group $G \in {\mathcal C}$ admits a weak $\mathcal Z$-structure. Moreover, the corresponding weak $\mathcal Z$-boundary has the shape of either a circle or a Hawaiian earring depending on whether $G$ is a virtually surface group or not.
\end{theorem}

%%%%%%%%%%%%%%%%%%%%%%%%%%%%%%%%%%%%%%%%%%%%%%%%%%%%%%%%%%%%%%%%%%%%%
\section{Preliminaries}
%%%%%%%%%%%%%%%%%%%%%%%%%%%%%%%%%%%%%%%%%%%%%%%%%%%%%%%%%%%%%%%%%%%%%
\indent Given a non-compact (strongly) locally finite CW-complex $Y$, a {\it
proper ray} in $Y$ is a proper map $\omega : [0, \infty) \lga Y$. Recall that a proper map is a
map with the property that the inverse image of every compact subset
is compact. We say that two proper rays $\omega, \omega'$ {\it define the same end}
if their restrictions to the natural numbers $\omega|_{\N},
\omega'|_{\N}$ are properly homotopic. This equivalence relation
gives rise to the notion of {\it end determined by $\omega$} as the
corresponding equivalence class, as well as the space of ends ${\mathcal E}(Y)$ of $Y$ as a compact totally
disconnected metrizable space (see \cite{BQ,Geo}). The CW-complex $Y$ is {\it
semistable} at the end determined by $\omega$ if
any other proper ray defining the same end is in fact properly
homotopic to $\omega$; equi\-va\-lently,
if the fundamental pro-group $pro-\pi_1(Y, \omega)$ is pro-isomorphic to
a tower of groups with surjective bonding homomorphisms (see \cite[Prop. 16.1.2]{Geo}). Recall that the homotopy pro-groups $pro-\pi_n(Y, \omega)$ are represented by the inverse
sequences (tower) of groups
\[
\pi_n(Y, \omega(0))
\stackrel{\phi_1}{\lgaf} \pi_n(Y - C_1, \omega(t_1))
\stackrel{\phi_2}{\lgaf} \pi_n(Y - C_2, \omega(t_2))
\lgaf \cdots
\]
where $C_1
\sub C_2 \sub \cdots \sub Y$ is a filtration of $Y$ by compact subspaces, $\omega([t_i,\infty)) \sub Y - C_i$
and the bonding homomorphisms $\phi_i$ are
induced by the inclusions and basepoint-change isomorphisms. One can show the independence with respect to the filtration. Also, properly homotopic base rays yield pro-isomorphic homotopy pro-groups $pro-\pi_n$, for all $n$. If $Y$ is semistable at each end then we will simply say that $Y$ is semistable at infinity, and in this case two proper rays representing the same end yield the same (up to pro-isomorphism) homotopy pro-groups $pro-\pi_n$. We refer to \cite{Geo,MarSe} for more details.\\

\indent Given a CW-complex $X$, with $\pi_1(X) \cong G$, we will
denote by $\widetilde{X}$ the universal cover of $X$, constructed as
prescribed in (\cite{Geo}, $\S 3.2$), so that $G$ is acting freely
on the CW-complex $\widetilde{X}$ via a cell-permuting left action
with $G \backslash \widetilde{X} = X$. The number of ends of an
(infinite) finitely generated group $G$ represents the number of ends
of the (strongly) locally finite CW-complex $\widetilde{X}^1$, for some (equivalently any)
CW-complex $X$ with $\pi_1(X) \cong G$ and with finite $1$-skeleton,
which is either $1, 2$ or $\infty$ (finite groups have $0$ ends
\cite{Geo,SWa}). If $G$ is finitely presented, then $G$ is {\it
semistable at infinity} if the (strongly) locally finite CW-complex $\widetilde{X}^2$ is so, for some (equivalently, any)
CW-complex $X$ with $\pi_1(X) \cong G$ and with finite $2$-skeleton. Observe that any finite-dimensional locally finite CW-complex is strongly locally finite, see \cite{Geo}.\\

\indent The following result will be crucial for the proof of the main result in this paper.
\begin{proposition} \label{aspherical} Let $\mathcal P = \langle X; r \rangle$ be a finite presentation of a torsion free group with a single (cyclically reduced) relator $r \in F(X)$, and consider the associated $2$-dimensional CW-complex $K_{\mathcal P}$. Then, the (contractible) universal cover $\widetilde{K}_{\mathcal P}$ is properly aspherical at infinity, i.e., for any choice of base ray, the homotopy pro-groups $pro-\pi_n(\widetilde{K}_{\mathcal P}) = 0$ are pro-trivial for $n \geq 2$. Furthermore, the fundamental pro-group $pro-\pi_1(\widetilde{K}_{\mathcal P})$ is pro-(finitely generated free).
\end{proposition}
Observe that the universal cover $\widetilde{K}_{\mathcal P}$ above is already known to be contractible (see \cite{DV}) and semistable at infinity (see \cite{MiTsch}), and hence its homotopy pro-groups do not depend (up to pro-isomorphism) on the choice of the base ray.
\begin{remark} \label{Freiheitssatz} Recall that the (finite) $2$-dimensional CW-complex $K_{\mathcal P}$ associated to $\mathcal P$ is constructed as follows. The $0$-skeleton consists of a single vertex and the $1$-skeleton $K^1_{\mathcal P}$ consists of a bouquet of circles, one for each element of the basis $x_i \in X$, all of them sharing the single vertex in $K_{\mathcal P}$. Finally $K_{\mathcal P}$ is obtained from $K^1_{\mathcal P}$ by attaching a $2$-cell $d$ via a PL map $S^1 \lga K^1_{\mathcal P}$ which spells out the single relator $r$. Note that every lift in the universal cover $\tilde{d} \sub \widetilde{K}_{\mathcal P}$ of the $2$-cell $d \sub K_{\mathcal P}$ is a disk as $r$ is a cyclically reduced word. Moreover, by the Magnus' Freiheitssatz (see \cite{LS, Mag}) every subcomplex of the $1$-skeleton $K^1_{\mathcal P}$ not containing all the $1$-cells involved in the relator $r$ lifts in the universal cover $\widetilde{K}_{\mathcal P}$ to a disjoint union of trees.
\end{remark}
Proposition \ref{aspherical} follows immediately from the following lemma, which is an enhancement of \cite[Prop. 2.7]{CLQR2}.
\begin{lemma} \label{enhance} Let $\mathcal P = \langle X; r \rangle$ be a finite, torsion free group presentation with a single (cyclically reduced) relator $r \in F(X)$, and consider the associated $2$-dimensional CW-complex $K_{\mathcal P}$. Then, the universal cover $\widetilde{K}_{\mathcal P}$ is proper homotopy equivalent to another $2$-dimensional CW-complex $\widehat{K}_{\mathcal P}$ which has a filtration $\widehat{C}_1 \sub \widehat{C}_2 \sub \cdots \sub \widehat{K}_{\mathcal P}$ by finite contractible subcomplexes satisfying (for any choice of base ray):
\begin{itemize}
\item[(a)] The tower $\{1\} \leftarrow \pi_1(\widehat{K}_{\mathcal P} - \widehat{C}_1) \leftarrow \pi_1(\widehat{K}_{\mathcal P} - \widehat{C}_2) \leftarrow \cdots$ consists of finitely generated free groups of increasing rank, with the bonding maps being the obvious projections, and
\item[(b)] The tower $\{1\} \leftarrow \pi_n(\widehat{K}_{\mathcal P} - \widehat{C}_1) \leftarrow \pi_n(\widehat{K}_{\mathcal P} - \widehat{C}_2) \leftarrow \cdots$ is the trivial tower, $n \geq 2$.
\end{itemize}
\end{lemma}
\begin{remark} In fact, the proper homotopy equivalence in the statement of Lemma \ref{enhance} can be replaced by a ``strong" proper homotopy equivalence, i.e., a (possibly infinite) sequence of internal collapses and/or expansions, carried out in a proper fashion. See \cite{CLQR2} for more details.
\end{remark}
\begin{proof} Indeed, the proof of this lemma is that of \cite[Prop.2.7]{CLQR2}, only that now we extend it, by taking a closer look, so that it covers part (b) here. The proof there goes by induction on the length of the relator $r \in F(X)$ in such a presentation $\mathcal P = \langle X; r \rangle$. It consists of a simultaneous double induction argument keeping track of two possible cases, depending on whether there is a generator in $X$ whose exponent sum in $r$ is zero or not, see \S 3 and \S 4 in \cite{CLQR2} respectively.\\
\indent In the first case (\S 3 in \cite{CLQR2}), one shows that the induction lies
on the fact that $K'_{\mathcal P}$, an intermediate cover of the CW-complex $K_{\mathcal P}$, is made out, up to homotopy, of blocks $K_{\mathcal P'}$, where $\mathcal P'$ satisfies the inductive hypothesis. In fact, its universal cover $\widetilde{K_{\mathcal P'}}$ is being slightly altered (within their proper homotopy type) to a CW-complexes $\widehat{K}_{\mathcal P'}$ so that their copies can be assembled together resulting into a new CW-complex $\widehat{K}_{\mathcal P}$ strongly proper homotopy equivalent to the universal cover of $K_{\mathcal P}$. This new CW-complex $\widehat{K}_{\mathcal P}$ consists of copies of the various CW-complexes $\widehat{K}_{\mathcal P'}$ above, glued together along trees (which were already present in the universal cover of $K_{\mathcal P}$, that correspond	to the intersections of the different copies of $K_{\mathcal P'}$ and whose existence is a consequence of the Magnus' Freiheitssatz, see Remark \ref{Freiheitssatz}.\\
\indent The desired filtration for $\widehat{K}_{\mathcal P}$ is then the result of assembling the filtrations we encounter on the various complexes $\widehat{K}_{\mathcal P'}$, which already have one by induction, as we grow towards infinity. This can be
carefully done in such a way that if two of these CW-complexes $\widehat{K}_{\mathcal P'}$
meet along a tree inside $\widehat{K}_{\mathcal P}$ then each of the members of the
corresponding filtration for each of them intersects that tree in a connected subtree.\\
\indent Finally, given a compact subset $\widehat{C}_n \sub \widehat{K}_{\mathcal P}$ from this resulting filtration, the generalized van-Kampen theorem yields that the fundamental
group $\pi_1(\widehat{K}_{\mathcal P} - \widehat{C}_n)$ is the free product of a free group
together with the various $\pi_1(\widehat{K}_{\mathcal P'} - \widehat{C}'_n)$ (finitely generated free by induction), where $\widehat{C}'_n = \widehat{K}_{\mathcal P'} \cap \widehat{C}_n \neq \emptyset$.\\
\indent The novelty here consists of adding part (b) of the statement to the induction hypothesis,	and observing that each neighborhood of infinity of the form
$U = \widehat{K}_{\mathcal P} - \widehat{C}_n$ is an assembly of the various neighborhoods
of infinity $U' = \widehat{K}_{\mathcal P'} - \widehat{C}'_n$ (with $\widehat{C}'_n = \widehat{K}_{\mathcal P'} \cap \widehat{C}_n \neq \emptyset$) together with all those (contractible) copies $\widehat{K}_{\mathcal P'} \sub \widehat{K}_{\mathcal P}$  which
do not intersect $\widehat{C}_n$. Moreover, if two of the neighborhoods of infinity $U'$
(corresponding to two different copies of $\widehat{K}_{\mathcal P'}$) intersect inside
$\widehat{K}_{\mathcal P}$, then they do it along the various components of $T - \widehat{C}'_n$, where $T \sub \widehat{K}_{\mathcal P}$ is the corresponding tree along which those copies of $\widehat{K}_{\mathcal P'}$ are glued together inside $\widehat{K}_{\mathcal P}$. This way, the universal cover $\widetilde{U}$ of
$U = \widehat{K}_{\mathcal P} - \widehat{C}_n$ is the result of putting together the universal covers $\widetilde{U}'$ of the various neighborhoods of infinity $U' = \widehat{K}_{\mathcal P'} - \widehat{C}'_n$ glued along connected subtrees, together with all those copies $\widehat{K}_{\mathcal P'} \sub \widehat{K}_{\mathcal P}$ which do not intersect $\widehat{C}_n$, each one glued to the rest along a copy of the corresponding tree from the construction indicated above. Thus, the induction hypothesis guarantees that each $\widetilde{U}'$ is a contractible CW-complex and hence part (b) follows for $\widehat{K}_{\mathcal P}$.\\
\indent As for the second case (\S 4 in \cite{CLQR2}), in which there is no generator in $X$ whose exponent sum in $r$ is zero, the proof goes somehow the other way around. An auxiliary CW-complex $K_{\mathcal P'}$ is built. For such $K_{\mathcal P'}$ the induction hypothesis
applies since it has a generator whose exponent sum is zero in the presentation $\mathcal P'$, which lies under the inductive hypothesis for the previous case (\S 3 in \cite{CLQR2}). As above, its universal cover can be slightly altered (within its proper homotopy type)
to a new CW-complex $\widehat{K}_{\mathcal P'}$ which is made out of blocks, corresponding
to copies of our candidates for the CW-complex $\widehat{K}_{\mathcal P}$ in question, glued together along copies of the real line. Given an appropriate filtration
$\widehat{C}'_n \sub \widehat{K}_{\mathcal P'}$ by compact subsets (provided by the induction hypothesis) satisfying the required properties for $\widehat{K}_{\mathcal P'}$, one can get the desired filtration on each copy $\widehat{K}_{\mathcal P}$ inside $\widehat{K}_{\mathcal P'}$ simply by considering the intersections $\widehat{C}_n = \widehat{K}_{\mathcal P} \cap \widehat{C}'_n$. Observe that this procedure may yield different choices for the desired filtration on each of those copies of $\widehat{K}_{\mathcal P}$, but they all satisfy the required properties (a)-(b). Indeed, by induction, each neighborhood of infinity in $\widehat{K}_{\mathcal P'}$ of the form $U' = \widehat{K}_{\mathcal P'} - \widehat{C}'_n$ has finitely generated free fundamental group and trivial higher homotopy groups. From here, the argument is similar to the one given above, concluding that the corresponding neighborhoods of infinity $U = \widehat{K}_{\mathcal P} - \widehat{C}_n$ in each copy $\widehat{K}_{\mathcal P}$ inside $\widehat{K}_{\mathcal P'}$ behave in the same way (as each $\pi_1(U)$ is now a free factor of $\pi_1(U')$).
\end{proof}
\indent A tower of groups $F \equiv (\{1\} \lgaf F_1 \lgaf F_2 \lgaf \cdots)$ consisting of finitely generated free groups of non-decreasing rank and the obvious projections as bonding maps will be said to be ``telescopic" (or of telescopic type). One can always associate to any given telescopic tower a $1$-ended locally-finite (simply connected) $2$-dimensional CW-complex $Y_m$, $0 \leq m \leq \infty$, whose fundamental pro-group realizes that telescopic tower as follows. Set $Y_0 = \{\ast\} \times [0, \infty)$ (a copy of ${\R}_+$). Assume $Y_n$ constructed, $n \in {\N} \cup \{0\}$. Then, $Y_{n+1}$ consists of the proper wedge of $Y_n$ and a copy $S^1 \times [n, \infty) \cup D^2 \times \{n\}$ of ${\R}^2$ attached along $Y_0$. Finally, we set $Y_\infty = \cup_{n \geq 0} Y_n$. Indeed, one can easily check that for some $0 \leq m \leq \infty$ and some filtration $\{J_n\}_{n \geq 1}$ of $Y_m$, there is a pro-isomorphism $\psi = \{\psi_n\}_{n \geq 1} : pro-\pi_1(Y_m) \lga F$, where each $\psi_n : \pi_1(Y_m - J_n) \lga F_n$ is an isomorphism between finitely generated free groups. Observe that the proper homotopy type of $Y_m$ can be represented by a subpolyhedron of ${\R}^3$, see the figure below.
%\vspace{-5mm}

\begin{figure} [h!]
\centerline{\psfig{figure=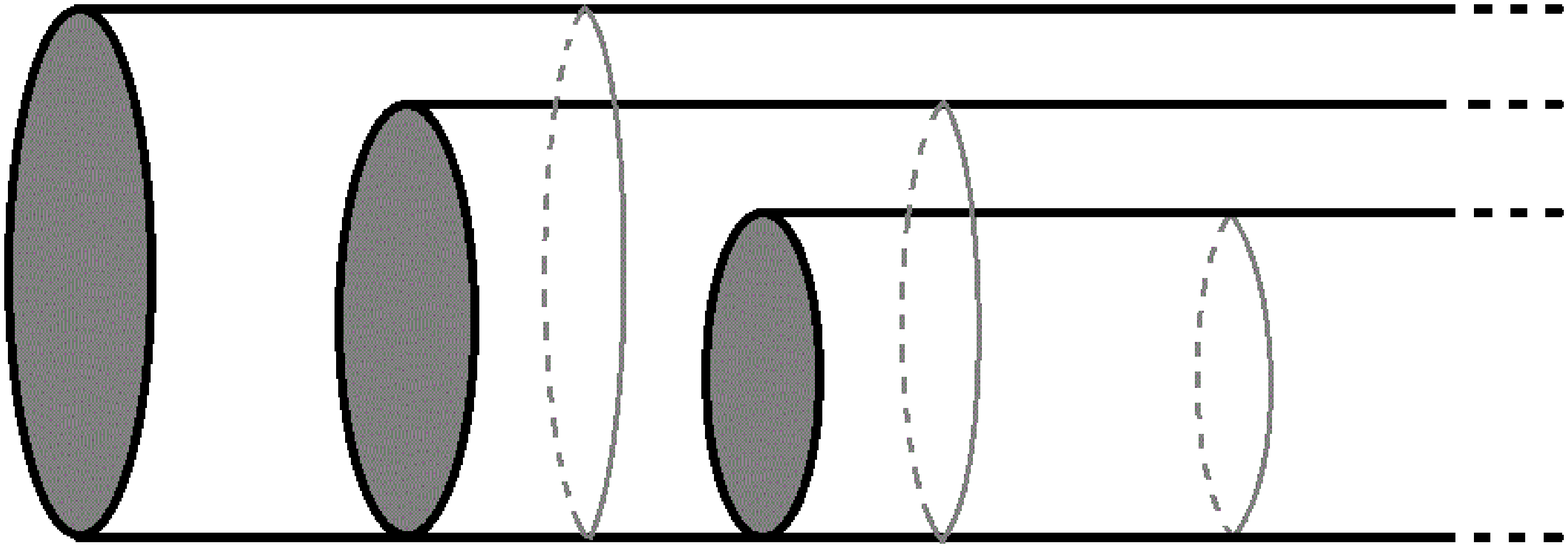,height=3cm,width=8cm}}
\label{figure} \caption{}
\vspace{-2mm}
\end{figure}

\begin{corollary} \label{flute} With the above notation, in the $1$-ended case the universal cover $\widetilde{K}_{\mathcal P}$ is proper homotopy equivalent to either $Y_1 (= {\R}^2)$ or $Y_\infty$.
\end{corollary}
\begin{proof} According to the above, by Lemma \ref{enhance} (a), there is some $0 \leq  m \leq  \infty$ and a pro-isomorphism $\psi = \{\psi_n\}_n : pro-\pi_1(Y_m) \lga pro-\pi_1(\widehat{K}_{\mathcal P})$, with each $\psi_n : \pi_1(Y_m - J_n) \lga \pi_1(\widehat{K}_{\mathcal P} - \widehat{C}_n)$ being an isomorphism between finitely generated free groups. Moreover, by Lemma \ref{enhance} (b) and \cite[Prop. 3.3]{CMQ}, there is a proper map $f : Y_m \lga \widehat{K}_{\mathcal P}$ inducing the pro-isomorphism $\psi$; in fact, $f$ is a weak proper homotopy equivalence, as $Y_m$ is clearly properly aspherical at infinity as well, and hence $f$ induces pro-isomorphisms between all the homotopy pro-groups. Therefore, by the corresponding proper Whitehead theorem (see \cite[Thm. 5.5.3]{EdHs} or \cite[$\S$ 8]{BQ}, for instance) $f$ is in fact a proper homotopy equivalence.\\
\indent It remains to show that $m = 1$ or $\infty$. For this, observe that $m > 0$ since otherwise $\widehat{K}_{\mathcal P}$ (and hence $\widetilde{K}_{\mathcal P}$) would be proper homotopy equivalent to a $3$-manifold with a single plane on its boundary (as $Y_0 = {\R}_+$ thickens to a $3$-dimensional half-space), which is not possible by \cite[Cor. 5.14]{CLQR_JPAA}. Furthermore, $Y_m$ (and hence $\widetilde{K}_{\mathcal P}$) must be proper homotopy equivalent to a $3$-manifold with boundary (by means of a regular neighborhood of the subpolyhedron of ${\R}^3$ in the figure above) which can only have either two or infinitely many plane boundary components, by \cite[Cor. 5.11, 5.14]{CLQR_JPAA}. The rest of the proof follows from this and the fact that the first option only occurs in the case of a virtually surface group, see \cite[Thm. 5.17]{CLQR_JPAA}.
\end{proof}
\begin{remark} In terms of \cite{CLQR_JPAA}, every $1$-ended, torsion free one-relator group is proper $2$-equivalent to either ${\Z} \times {\Z}$ or ${\F}_2 \times {\Z}$ by \cite[Thm 5.1]{CLQR_JPAA}, as one relator groups are properly $3$-realizable, see \cite{CLQR2, LasRoy}; in fact, given a presentation $\mathcal P$ as above, the universal cover of $K_{\mathcal P}$ itself is proper homotopy equivalent to a $3$-manifold (by considering a regular neighborhood of the above subpolyhedron in ${\R}^3$) with no need to take wedge with a single $2$-sphere, thus answering in the affirmative a conjecture posed in \cite{CLQR2} (in the torsion free case). Observe that the third option ${\Z} \times {\Z} \times {\Z}$ from \cite[Thm 5.1]{CLQR_JPAA} is ruled out by \cite[Cor. 5.14]{CLQR_JPAA}, and the first option ${\Z} \times {\Z}$ only occurs in the case of a virtually surface group, by \cite[Thm. 5.17]{CLQR_JPAA}.
\end{remark}
%%%%%%%%%%%%%%%%%%%%%%%%%%%%%%%%%%%%%%%%%%%%%%%%%%%%%%%%%%%%%%%%%%%%%%%
\section{Proof of the main results}
%%%%%%%%%%%%%%%%%%%%%%%%%%%%%%%%%%%%%%%%%%%%%%%%%%%%%%%%%%%%%%%%%%%%%%%
The purpose of this section is to prove Theorems \ref{thm1} and \ref{thm2}. For this, we need the following previous result, which is a combination of other well known results.
\begin{lemma} \label{CS} Let $X$ be a locally finite $n$-dimensional (PL)CW-complex. If the following two conditions hold:
\begin{itemize}
\item[(a)] $X$ is inward tame, and
\item[(b)] For any choice of base ray, the fundamental pro-group $pro-\pi_1(X)$ is pro-(finitely generated free)
\end{itemize}
then the product $X \times I^{2n+5}$ admits a $\mathcal Z$-compactification, with $I = [0,1]$.
\end{lemma}
We recall that an ANR $X$ is {\it inward tame} if, for each neighborhood of infinity $N$ there exists a smaller neighbohood of infinity $N' \sub N$ so that, up to homotopy, the inclusion $N' \hookrightarrow N$ factors through a finite complex; equivalently, for every closed neighborhood of infinity $N$ there is a homotopy $H : N \times [0,1] \lga N$ with $H_0 = id_N$ and $\overline{H_1(N)}$ compact, see \cite{Gui3}. One can check that inward tameness is a proper homotopy invariant, and we may think of it as pulling the end of $X$ inside $X$ yielding some kind of finite domination at infinity. It is worth noticing that any $\mathcal Z$-compactifiable ANR must be inward tame, see \cite[Remark 3.8.13]{Gui3}. Also, observe that as a combination of the results in \cite{Fe} and \cite{Gui4}, there is an example of a locally finite $2$-dimensional polyhedron $X$ whose product $X \times I^9$ is $\mathcal Z$-compactifiable but $X$ itself is not.
\begin{proof} Let $I^\infty$ denote the Hilbert cube. It is well known that the product $Y = X \times I^\infty$ is a Hilbert cube manifold (see \cite{W,Ed}) which satisfies again properties (a) and (b) from the statement, as $I^\infty$ is compact and contractible and so $X$ and $Y$ are proper homotopy equivalent. In particular, $Y$ is inward tame. Moreover, the Chapman-Siebenmann obstructions for a Hilbert cube manifold admitting a $\mathcal Z$-compactification (\cite[Thms 3, 4]{CS}, see also \cite[\S 3.8.2]{Gui3}) vanish for $Y$ since $pro-\pi_1(Y)$ can be represented by an inverse sequence
\[
\pi_1(Y) \lgaf \pi_1(N_1) \lgaf \pi_1(N_2) \lgaf \cdots
\]
where $\{N_i\}_i$ is a nested cofinal sequence of neighborhoods of infinity in $Y$ with $\pi_1(N_i)$ a finitely generated free group, $i \geq 1$; in fact, each $N_i$ can be taken as a product $N_i = M_i \times I^\infty$, where $M_i$ is a neighborhood of infinity in $X$. Thus $Y = X \times I^\infty$ admits a $\mathcal Z$-compactification. Finally, the results in \cite{Fe} show that $X \times I^{2n+5}$ admits a $\mathcal Z$-compactification as well.
\end{proof}
We now proceed with the proof of the main results.
\begin{proof}[Proof of Theorem \ref{thm1}] Suppose a given torsion free finitely presented group $G$ admits a finite presentation $\mathcal P = \langle X ; r \rangle$ with a single (cyclically reduced) relator $r \in F(X)$. If $G$ is $2$-ended then $G$ must be the group of integers ${\Z}$ (see \cite[Thm. 5.12]{SWa}) which easily admits a weak $\mathcal Z$-structure just by adding two points as its boundary. Assume now $G$ is $1$-ended. Then, by Corollary \ref{flute}, the universal cover $\widetilde{K}_{\mathcal P}$ is proper homotopy equivalent to either the plane ${\R}^2$ or the locally finite subpolydedron of ${\R}^3$ shown in figure $1$, which are both easily shown to be inward tame, and hence so is $\widetilde{K}_{\mathcal P}$. On the other hand, Proposition \ref{aspherical} ensures condition (b) in Lemma \ref{CS} above. Therefore, the (contractible) CW-complex $\widetilde{K}_{\mathcal P} \times I^9$ admits a $\mathcal Z$-compactification. Observe that the proper, free and cocompact $G$ action on $\widetilde{K}_{\mathcal P}$ yields a proper, free and cocompact $G$ action on $\widetilde{K}_{\mathcal P} \times I^9$ in the obvious way, thus providing a weak $\mathcal Z$-structure on $G$ whose associated weak $\mathcal Z$-boundary has the shape of the $\mathcal Z$-boundary of a $\mathcal Z$-compactification of either the plane or the subpolyhedron shown in figure $1$, see \cite[Cor. 3.8.15]{Gui3}. In the case of the plane this $\mathcal Z$-boundary has the shape of a circle, and in the second case one can easily show that the corresponding $\mathcal Z$-boundary has the shape of a Hawaiian earring, as claimed.\\
\indent Finally, if $G$ is infinite ended then $G$ decomposes as a free product of groups (as $G$ is torsion free) by the Stallings's structure theorem (see [13,27]). Moreover, being $G$ a one-relator group, it follows from Grushko's theorem that $G$ is a free product of a free group and a one-relator group with at most one end. See [22, Prop. II.5.13] for details. Both factors admit a weak $\mathcal Z$-structure and hence so does their free product, by the proof of \cite[Thm. 2.9]{T}.
\end{proof}
Just as we did in section $\S 2$ with respect to the work in \cite{CLQR2}, a closer look at the proofs of \cite[Thm. 1.13]{LasRoy} and \cite[Prop. 1.18]{LasRoy} yields the following generalization of Proposition \ref{aspherical} and Lemma \ref{enhance} (in the $1$-ended case).
\begin{proposition} \label{aspherical2} Let $\mathcal P = \langle X ; R \rangle$ be a finite aspherical presentation of a torsion free, $1$-ended generalized one-relator group $G \in {\mathcal C}$, with each $r \in R$ being a cyclically reduced word in $F(X)$, and consider the associated $2$-dimensional CW-complex $K_{\mathcal P}$. Then, the (contractible) universal cover $\widetilde{K}_{\mathcal P}$ is properly aspherical at infinity, i.e., for any choice of base ray, the homotopy pro-groups $pro-\pi_n(\widetilde{K}_{\mathcal P}) = 0$ are pro-trivial for $n \geq 2$, and the fundamental pro-group $pro-\pi_1(\widetilde{K}_{\mathcal P})$ is pro-isomorphic to a telescopic tower
\end{proposition}
Thus, the proof of Theorem \ref{thm2} is the same as that of Theorem \ref{thm1} in the $1$-ended case.
\begin{remark}
It is worth pointing out that sometimes the strategy followed to prove that some classes of $1$-ended groups admit a weak $\mathcal Z$-structure includes showing that the fundamental pro-group is pro-(finitely generated free). Under semistability at infinity, this property about the fundamental pro-group amounts to saying that the groups under study are {\it properly $3$-realizable}, i.e., they can be realized by a finite $2$-dimensional CW-complex whose universal cover is proper homotopy equivalent to a $3$-manifold. See \cite[Thm. 1.2]{L}) and \cite[Thm. 5.22]{CLQR_JPAA}. The above is the case of this and other papers, see \cite{Gui1} for instance. At the time of writing it is unknown whether there is a relation between proper $3$-realizability and the existence of a weak $\mathcal Z$-structure.
\end{remark}

\end{document}